\documentclass[12pt]{amsart}

\usepackage{geometry}
\usepackage{wasysym}
\usepackage{amsfonts, amssymb, upgreek}
\usepackage{mathrsfs,mathcomp, textcomp}
\usepackage{hyperref}

\usepackage[matrix,arrow,curve]{xy}

\makeatletter
\@addtoreset{equation}{subsection}
\makeatother

\newcommand{\B}{{\mathbf B}}

\newcommand{\QQ}{\mathbb{Q}}

\newcommand{\ZZ}{\mathbb{Z}}
\newcommand{\PP}{\mathbb{P}}

\newcommand{\MMM}{{\mathscr{M}}}

\newcommand{\qW}{\operatorname{q}_{\operatorname{W}}}
\newcommand{\qQ}{\operatorname{q}}
\newcommand{\qq}{\mathbin{\sim_{\scriptscriptstyle{\QQ}}}}
\newcommand{\R}{\operatorname{R}}

\newcommand{\Proj}{\operatorname{Proj}}
\newcommand{\Bs}{\operatorname{Bs}}
\newcommand{\Pic}{\operatorname{Pic}}
\newcommand{\Cl}{\operatorname{Cl}}
\newcommand{\Clt}{\operatorname{T}}
\newcommand{\ct}{\operatorname{ct}}

\newcommand{\g}{\operatorname{g}}
\newcommand{\rP}{\operatorname{P}}
\newcommand{\di}{\operatorname{div}}

\newcommand{\lbraces}{\{}
\newcommand{\rbraces}{\}} 

\newcommand{\xref}[1]{{\rm \ref{#1}}}

\swapnumbers

\theoremstyle{plain}
\newtheorem{theorem}[subsection]{Theorem}
\newtheorem{lemma}[subsection]{Lemma}

\newtheorem{proposition}[subsection]{Proposition}
\newtheorem{proposition-definition}[subsection]{Proposition-Definition}
\newtheorem{stheorem}[equation]{Theorem}
\newtheorem{sproposition-definition}[equation]{Proposition-Definition}

\newtheorem{corollary}[subsection]{Corollary}
\newtheorem{scorollary}[equation]{Corollary}
\newtheorem{claim}[subsection]{Claim}

\newtheorem{slemma}[equation]{Lemma}
\newtheorem{sproposition}[equation]{Proposition}

\theoremstyle{definition}
\newtheorem{assumption-notation}[subsection]{Assumption and notation}

\newtheorem*{definition*}{Definition}
\newtheorem{sdefinition}[equation]{Definition}
\newtheorem{example-remark}[subsection]{Remark-Example}
\newtheorem{subexample-remark}[equation]{Remark-Example}

\newtheorem*{notation*}{Notation}

\newtheorem{construction-example}[subsection]{Construction-Example}
\newtheorem{sconstruction-example}[equation]{Construction-Example}

\theoremstyle{definition}

\newtheorem*{remark*}{Remark}

\renewcommand\labelenumi{{\rm (\roman{enumi})}}
\renewcommand\theenumi{{\rm (\roman{enumi})}}

\title{$\QQ$-Fano threefolds of Fano index thirteen}

\author{Yuri Prokhorov}

\address{
Steklov Mathematical Institute of Russian Academy of Sciences, Moscow, RUSSIA
\newline\indent
Department 
of Algebra, Faculty of Mathematics, Moscow State
University, Moscow 117234, RUSSIA
\newline\indent
Laboratory of Algebraic Geometry, SU-HSE, 
7 Vavilova Str., Moscow, 117312, RUSSIA
}
\email{prokhoro@mi-ras.ru}
\thanks{ The author was partially supported by the HSE University Basic Research Program}
\begin{document}

\begin{abstract}
We show that a non-toric $\QQ$-factorial terminal Fano threefold 
of Picard rank $1$ and Fano index $13$
is a weighted hypersurface of degree $12$ in $\PP(3,4,5,6,7)$.
\end{abstract}
\maketitle
\section{Introduction}
A \textit{$\QQ$-Fano threefold} is a projective variety $X$ of dimension $3$ with only terminal, $\QQ$-factorial
singularities whose anticanonical class
$-K_X$ is ample and Picard rank equals $1$. 
The \textit{genus} of $X$ is
\[
\g(X) :=\dim H^0(X,-K_X ) - 2.
\]
The \textit{Fano index} of a $\QQ$-Fano threefold $X$ is defined as follows:
\[
\qQ(X): = \max\{ q \mid -K_X\qq qA,\ \text{where $A$ is a Weil divisor on $X$}\}.
\]
A Weil divisor $A_X$ such that $-K_X\qq \qQ(X)\, A_X$ is called the 
\textit{fundamental divisor} of $X$. Note that $A_X$ is defined modulo 
$\QQ$-linear equivalence.
Similarly, one can define \textit{Fano-Weil index} of $X$:
\[
\qW(X): = \max\{ q \mid -K_X\sim qB,\ \text{where $B$ is a Weil divisor on $X$}\}.
\]
Obviously, $\qQ(X)=\qW(X)$ if the group $\Cl(X)$ is torsion free.

\begin{theorem}[\cite{Suzuki-2004}, \cite{P:2010:QFano}]
\label{thm:indices}
The Fano index of a $\QQ$-Fano threefold takes the following values
\[
\qQ(X)\in \{1,2,3,\dots, 8,9,11,13,17,19\}. 
\]
\end{theorem}

\begin{theorem}[\cite{P:2010:QFano}]
\label{thm:19-17}
Let $X$ be a $\QQ$-Fano threefold. 
\begin{enumerate}
\item 
If 
$\qQ(X)= 19$, then $X \simeq \PP(3, 4, 5, 7)$. 
\item 
If 
$\qQ(X) = 17$, then $X \simeq \PP(2, 3, 5, 7)$. 
\item 
If 
$\qQ(X) = 13$ and $\g(X) > 4$, then $X \simeq \PP(1, 3, 4, 5)$. 
\item 
If 
$\qQ(X) = 11$ and $\g(X) > 9$, then $X \simeq \PP(1, 2, 3, 5)$. 

\item 
If 
$\qQ(X)\ge 7$ and there are two distinct effective Weil divisors $A_1,\, A_2$ such that $A_1\qq A_2\qq A_X$, then $X \simeq \PP(1,1, 2, 3)$. 
\end{enumerate}
\end{theorem}

In this note we concentrate on the case $\qQ(X)=13$ and $X \not\simeq \PP(1, 3, 4, 5)$.
It is easy to see that a general hypersurface $X_{12}$ of degree $12$ in $\PP(3, 4, 5, 6, 7)$
belongs to this class \cite{Brown-Suzuki-2007j}. Moreover, 
it is known that any $\QQ$-Fano threefold $X$ with $\qQ(X)=13$ and 
$X \not\simeq \PP(1, 3, 4, 5)$ has the 
same 
Hilbert
series as $X_{12}$ \cite{GRD}.
Our main result is the following.

\begin{theorem} 
\label{thm:main}
Let $X$ be a $\QQ$-Fano threefold with $\qQ(X)=13$.
Assume that $X \not\simeq \PP(1, 3, 4, 5)$. 
Then $X$ is isomorphic to a hypersurface of degree $12$ in $\PP(3, 4, 5, 6, 7)$.
Moreover, the equation of $X$ can be written in one the following forms:
\begin{enumerate}
\renewcommand\labelenumi{{\rm (\alph{enumi})}}
\renewcommand\theenumi{{\rm (\alph{enumi})}}
\item \label{thm:maina}
$x_5x_7+x_4^3+ x_6^2+ x_3^4$, or
\item \label{thm:mainb}
$x_5x_7+x_4^3+ x_6^2$,
\end{enumerate}
where $x_k$ is a coordinate of degree $k$.
\end{theorem}

Results of this kind were known earlier: 
it was proved in \cite{P:2013-fano} and \cite{P:2016:QFano7}
that many specific $\QQ$-Fano threefolds of large Fano index are
hypersurfaces. For example, 
a $\QQ$-Fano threefold with 
$\qQ(X)=9$ and $A^3_X=1/20$ is isomorphic to $X=X_6\subset \PP(1,2,3,4,5)$ \cite{P:2013-fano}.
In a recent work \cite{Chen-Jiang:QFano2} it is proved that certain $\QQ$-Fano threefolds of Fano index $1$ and small volume 
are hypersurfaces. 

The idea of the proof is quite standard: we take general sections $\varphi_m\in H^0(X, mA_X)$
for $m=3,\dots,7$ and consider the map 
\[
\Psi: X \dashrightarrow \PP(3,4,5,6,7) 
\]
given by $(\varphi_3,\dots, \varphi_7)$. The crucial step in the proof is to show that 
$\Psi$ is birational onto its image (Corollary~\ref{cor:bir}). This can be done by considering certain 
Sarkisov link (Theorem~\ref{thm:|6A|}).

\section{Preliminaries}
\subsection{Notation.}
We work over an algebraically closed field $\Bbbk$ of characteristic $0$.
We employ the following notation.
\begin{itemize}
\item 
$\B(X)$ is the basket of singularities of a terminal threefold $X$ (see \cite{Reid:YPG});
\item 
$\Cl(X)$ denotes the Weil divisor class group of a normal variety $X$;
\item
$\Clt(X)$ is the torsion subgroup of $\Cl(X)$;
\end{itemize}

The Hilbert series of a $\QQ$-Fano threefold $X$ is the formal power series
\[
\mathrm{p}_X(t)= \sum_{m\ge 0} \dim H^0 (X, mA_X )\, t^m.
\]
It is computed by using the orbifold Riemann-Roch formula \cite{Reid:YPG}.
One can see that $\mathrm{p}_X(t)$ determines discrete invariants of $X$ such as 
degree $A_X^3$, genus $\g(X)$, basket of singularities $\B(X)$, etc \cite{Fletcher:invertingRR}.
\begin{proposition}[see e.g. {\cite{GRD}}]
\label{prop:QFano13}
Let $X$ be a $\QQ$-Fano threefold with $\qQ(X)=13$ and $\g(X)\le 4$. 
Then the following holds:
\begin{equation}
\B(X)=(2, 3, 3, 5, 7),\qquad A_X^3=\frac{1}{210},\qquad \g(X)=4,
\end{equation} 
\begin{equation}
\label{prop:QFano13b}
\mathrm{p}_X(t)= 
\frac{1-t^{12}}{\prod_{m=3}^7(1-t^m)}
=
1
+ t^3
+ t^4
+ t^5
+ 2\*t^6
+ 2\*t^7
+ 2\*t^8
+ 3\*t^9
+ 4\*t^{10}
+ 
\cdots
\end{equation}
\end{proposition}

\begin{corollary}
\label{cor:dim}
One has $|A_X|=|2A_X|=\varnothing$.
For $k=3$, $4$, and $5$ the linear system $|kA_X|$
consists of a \textup(unique\textup) prime divisor. 
For $k=6$ and $7$ the linear system $|kA_X|$ is a pencil without 
fixed components.
\end{corollary}
\begin{proof}
It follows from  \eqref{prop:QFano13b} that $|A_X|=|2A_X|=\varnothing$,
$\dim |3A_X|=\dim |4A_X|=\dim |5A_X|=0$, and $\dim |6A_X|=\dim |7A_X|=1$.
\end{proof}

\begin{lemma}
\label{lemma:tor}
Let $X$ be a $\QQ$-Fano threefold.
\begin{enumerate}
\item \label{lemma:tor-a}
If $\qQ(X)\ge 8$ or $\qQ(X)=6$, then the group $\Cl(X)$ is torsion free.
\item 
If $\qQ(X)\ge 4$ and $|\Clt(X)|\ge 3$, then one of the following holds: 
\begin{enumerate}
\item \label{lemma:tor-ba}
$\qQ(X)=5$, $|\Clt(X)|= 3$, $\B(X)=(2, 9, 9)$, $A_X^3=1/18$, $\g(X)=2$;
\item \label{lemma:tor-bb}
$\qQ(X)=4$, $|\Clt(X)|= 3$, $\B(X)=(9, 9)$, $A_X^3=1/9$, $\g(X)=3$;
\item \label{lemma:tor-bc}
$\qQ(X)=4$, $|\Clt(X)|=5$, $\B(X)=(5^4)$, $A_X^3=1/5$, $\g(X)=5$.
\end{enumerate}
\end{enumerate}
\end{lemma}
\begin{proof}
For $\qQ(X)\ge 9$ the assertion follows from 
\cite[Proposition~3.6]{P:2010:QFano}
and for $\qQ(X)\ge 5$ it follows from \cite[Propositions~3.3-3.4]{P:2019:rat:Q-Fano}.
For $\qQ(X)=4$ one can use the algorithm 
of \cite[\S~3]{P:2019:rat:Q-Fano} (see also \cite{Caravantes2008})
to find out the possibilities~\ref{lemma:tor-bb} and~\ref{lemma:tor-bc}.
\end{proof}

\section{Sarkisov link}
\subsection{}
Let $X$ be a $\QQ$-Fano threefold of Fano index $q=\qQ(X)$.
For simplicity, we assume that the group $\Cl(X)$ is torsion free.
Everywhere throughout this paper by $A=A_X$ we denote
the fundamental divisor of $X$, i.e. the positive generator of $\Cl(X)\simeq\ZZ$.
Then $-K_X=qA_X$.

We assume that there exists the following Sarkisov link:
\begin{equation}
\label{diagram-main}
\vcenter{
\xymatrix@C=19pt{
&\tilde{X}\ar@{-->}[rr]^{\chi}\ar[dl]_{f} && \bar{X}\ar[dr]^{\bar f}
\\
X &&&&\hat{X}
} 
}
\end{equation}
where $\chi$ is an isomorphism in codimension $1$,
the varieties $\tilde{X}$ and $ \bar{X}$ have only terminal $\QQ$-factorial singularities,
$\uprho (\tilde{X})=\uprho (\tilde{X})=2$, and
$\bar{f}: \bar{X}\to \hat{X}$ is an extremal $K_{\bar{X}}$-negative Mori contraction.

\subsection{}
Denote the $f$-exceptional divisor by $E$.
In what follows, for a divisor (or a linear system) $D$ on $X$
by $\tilde D$ and $\bar D$ we denote
proper transforms of $D$
on $\tilde{X}$ and $\bar{X}$, respectively.
If $|kA_X|\neq \varnothing$, we put $\MMM_k:=| kA_X|$.
By $M_k$ we denote
a general member of $\MMM_k$.
We can write
\begin{equation} 
\label{equation-1}
\begin{array}{lll}
K_{\tilde{X}} &\qq & f^*K_X+\alpha E,
\\[2pt]
\bar\MMM_k&\qq& f^*\MMM_k- \beta_k E,
\end{array}
\end{equation}
where $\alpha \in \QQ_{>0}$, $\beta \in \QQ_{\ge0}$.

Since the group $\Cl( \bar{X})$ is torsion free, the relations~\eqref{equation-1} 
give us
\begin{equation} 
\label{eq:12-divisors}
k K_{\bar{X}}+q \bar\MMM_k \sim (k \alpha -q \beta_k) E,
\end{equation}
where $k \alpha -q \beta_k$ is an integer.

\label{case-bir}
If the contraction $\bar{f}$ is birational, then $\hat{X}$ is a $\QQ $-Fano threefold.
In this case, denote by $\bar{F}$ the
$\bar{f}$-exceptional divisor, by
$\tilde F \subset \tilde{X}$ its proper transform, $F:=f(\tilde F)$, and
$\hat{q}:=\qQ (\hat{X})$. 
Let $A_{\hat{X}}$ be a fundamental divisor on $\hat{X}$.
Write
\[
F\sim d A_X,\qquad \hat{E} \qq eA_{\hat{X}},\qquad \hat\MMM_k \qq s_kA_{\hat{X}},
\]
where $d,\, e\in \ZZ_{\ge 0}$,\ $s_k \in \ZZ_{\ge0}$. 
Thus $s_k=0$ if and only if 
$\dim \MMM_k=0$ and $\bar{M}_k=\bar{F}$
(i.e. the unique element $M_k$ of the linear system $\bar\MMM_k$ is the $\bar{f}$-exceptional divisor).

\label{case-nonbir}
If the contraction $\bar{f}$ is not birational, we denote by
$\bar{F}$ a general geometric fiber.
Then $\bar{F}$ is either a smooth rational curve or a del Pezzo surface.
The image of the restriction map $\Cl(\bar{X}) \to \Pic (\bar{F})$ is isomorphic to $\ZZ$.

Let $\Xi$ be its ample generator.
As above, we can write
\[
-K_{\bar{X}} |_{\bar{F}} =-K_{\bar{F}} \sim \hat{q}\Xi,\qquad \bar{E} |_{\bar{F}} \sim e \Xi,\qquad \bar\MMM_k |_{\bar{F}} \sim s_k \Xi,
\]
where $\hat{q}\in \{1,2,3\}$, $e\in \ZZ_{\ge 0}$, and $s_k \in \ZZ_{\ge0}$.
In this case, $\bar f(\bar E)=\hat X$, i.e. $\bar E$ is $\bar f$-ample \cite[Claim~4.6]{P:2010:QFano}, hence $e>0$.

Regardless of whether $\bar{f}$ is birational or not,
from \eqref{eq:12-divisors} we obtain
\begin{equation} 
\label{eq:main}
k \hat{q}=q s_k+(q \beta_k-k \alpha) e.
\end{equation}

\subsection{}
Suppose that the morphism $\bar{f}$ is birational.
Then $\bar{E}\neq \bar{F}$ (see e.g. \cite[Claim~4.6]{P:2010:QFano}), hence $e>0$.

\begin{slemma}
\label{lemma:torsion} 
$|\Clt(\hat{X})|=d/e$. 
\end{slemma}

\begin{proof}
Follows from obvious isomorphisms
\[
\ZZ/d\ZZ\simeq \Cl(X)/(F\cdot \ZZ) \simeq \Cl( \bar{X})/(\bar{F}\cdot \ZZ\oplus\bar{E}\cdot\ZZ)
\simeq \Cl(\hat{X})/(\hat{E}\cdot\ZZ)
\]
and $\Cl(\hat{X})/(\Clt(\hat{X})\oplus \hat{E}\cdot\ZZ) \simeq \ZZ/e\ZZ$.
\end{proof}

\begin{scorollary}
\label{cor:torsion}
If $X$ is a $\QQ$-Fano threefold with $\qQ(X)=13$ and $X \not\simeq \PP(1, 3, 4, 5)$,
then $|\Clt(\hat{X})|\ge 3/e$.
\end{scorollary}
\begin{proof}
Since $|A_X|=|2A_X|=\varnothing$,\ $d\ge 3$. 
\end{proof}

Similar to \eqref {equation-1},
write
\[
K_{\bar{X}}\qq \bar{f}^*K_{\hat{X}}+b \bar{F}, \quad
\bar\MMM_k\qq \bar{f}^*\hat\MMM_k-\gamma_k \bar{F}, \quad
\bar{E}\qq \bar{f}^*\hat{E}-\delta \bar{F}.
\]
This gives us
\begin{equation*}
\begin{array}{lll}
s_kK_{\bar{X}}+\hat{q}\bar\MMM_k &\sim & (b s_k- \hat{q}\gamma_k) \bar{F},
\\[2pt]
eK_{\bar{X}}+\hat{q}E &\sim & (be-\hat{q}\delta) \bar{F}.
\end{array}
\end{equation*}
Taking proper transforms of these relations to $X$, we obtain
\begin{equation*} 
\begin{array}{lll}
-s_kq+k \hat{q}&=& (b s_k- \hat{q}\gamma_k) d,
\\[2pt]
-eq &=& (be-\hat{q}\delta) d.
\end{array}
\end{equation*}
If $\Cl(\hat{X})\simeq\ZZ$, then $e=d$ by Lemma \ref{lemma:torsion} and so
\begin{equation} 
\label{eq:b-gamma-delta-1}
\begin{array}{lll}
e \gamma_k &=& s_k \delta-k,
\\[2pt]
eb &=& \hat{q}\delta-q.
\end{array}
\end{equation}

\begin{slemma}
\label{lemma:g}
If $\alpha<1$, then $\g(X)\le \g(\hat{X})$.
\end{slemma}

\begin{proof}
Let $D\in |-K_X|$. Then $K_X+D\sim 0$ and
\[
K_{\tilde{X}}+\tilde D +a E\sim f^*(K_X+D)\sim 0,
\]
where $a$ is an integer because $K_X+D$ is Cartier.
By \eqref{equation-1}\ $a>-\alpha$, hence $a\ge 0$.
Thus for any $D\in |-K_X|$ there exists $\tilde D +a E\in |-K_{\tilde{X}}|$
such that $f_*(\tilde D +a E)=D$. This shows that $\dim |-K_{\tilde{X}}|=\dim |-K_{X}|$.
Then obviously,
\[
\g(\hat{X})+1= \dim |-K_{\hat{X}}|\ge \dim |-K_{ \bar{X}}|= \dim |-K_{\tilde{X}}|=\dim |-K_{X}|=\g(X)+1.
\qedhere
\]
\end{proof}

\subsection{}
A link of the form \eqref{diagram-main} can be constructed 
easily by using the following construction. 

\begin{sdefinition}
Let $X$ be a threefold with at worst terminal $\QQ$-factorial singularities.
An \emph{extremal blowup} of $X$ is  a birational morphism $f: \tilde{X} \to X$ 
such that 
$\tilde{X}$ has only terminal $\QQ$-factorial singularities and $\uprho(\tilde{X}/X)=1$.
\end{sdefinition}
Note that in this situation the divisor $-K_{\tilde X}$ is $f$-ample.

\begin{sproposition}[see e.g. 
{\cite{Corti95:Sark}} or {\cite[4.5.1]{P:G-MMP}}]
Let $X$ be a threefold with at worst terminal $\QQ$-factorial singularities 
and let $\MMM$ be a linear system on $X$ without fixed components.
Let $c:=\operatorname {ct} (X, \MMM)$ be the canonical threshold of the pair $(X, \MMM)$.
Then there exists an extremal blowup $f: \tilde{X} \to X$ 
that is crepant with respect to $K_X+c \MMM$.
\end{sproposition}

\begin{stheorem}
\label{thm:link}
Let $X$ be a $\QQ$-Fano threefold.
Let $\MMM$ be a linear system on $X$ without fixed components and 
let $c:=\operatorname {ct} (X, \MMM)$.
Assume that $-(K_X+c \MMM)$ is ample.
Then any extremal blowup $f: \tilde{X} \to X$  which is crepant with respect to $K_X+c \MMM$
can be extended to a Sarkisov link \eqref{diagram-main}.
\end{stheorem}

\begin{proof}
Run the log minimal model program on $\tilde{X}$ with respect to $K_{\tilde{X}}+c \tilde\MMM$
(see e.g. \cite[4.2]{Alexeev:ge} or \cite[12.2.1]{P:G-MMP}).
\end{proof}

\section{One particular Sarkisov link}
Throughout this section, let 
$X$ be a $\QQ$-Fano threefold with $\qQ(X)=13$ such that $X \not\simeq \PP(1, 3, 4, 5)$. 
Let 
$\rP_2$ (resp. $\rP_5$, $\rP_7$) be the unique point of index $2$ (resp., $5$, $7$) on $X$;
note that $\rP_2$, $\rP_5$, $\rP_7$ are cyclic quotient singularities.

Recall that for a cyclic quotient terminal threefold singularity $X\ni P$ of type $\frac1r (a, r-a,1)$, 
there exists exactly one extremal blowup $f:\tilde{X} \to X\ni P$ with center $P$; this is the weighted 
blowup with weights $\frac1r (a, r-a,)$ \cite{Kawamata:Div-contr}.
We call such an extraction the \textit{Kawamata blowup} of  $X\ni P$.

\begin{theorem}
\label{thm:|6A|}
Let $f:\tilde{X}\to X$
be the Kawamata blowup of $\rP_5$.
Then $f$ can be extended to a Sarkisov link \eqref{diagram-main} with $\hat{X}= \PP(1,1,2,3)$.
\end{theorem}

\begin{proposition}
\label{prop:|6A|}
The construction of Theorem~\xref{thm:link} with $\MMM:=|6A_X|$ 
produces a link \eqref{diagram-main} such that one of the following holds:
\begin{enumerate}
\item
\label{prop:|6A|:P(1,2,3,5)}
$f$
is  the Kawamata blowup of $\rP_2$ and $\hat{X}=\PP(1,2,3,5)$, or

\item 
\label{prop:|6A|:P(1,1,2,3)}
$f$
is the Kawamata blowup of $\rP_5$ and $\hat{X}= \PP(1,1,2,3)$.
\end{enumerate}
Moreover, $\ct(X, |6A_X|)=1/2$.
\end{proposition}
Note that here we do not assert that \textit{both} possibilities of Proposition~\ref{prop:|6A|} occur.
\begin{proof}
Put $\MMM:=|6A_X|$ and $c:=\ct(X,\MMM)$.
Near the point of index $5$ we have $A_X\sim \MMM \sim -2K_X$.
If $f':X'\to X$ is the Kawamata blowup of $\rP_5$,
then in \eqref{equation-1} one has $\alpha'=1/5$ and $\beta_6'\equiv 2\alpha' \mod \ZZ$.
Therefore, 
$c\le \alpha' /\beta_6' = 1/2$.
Now let $f:\tilde{X}\to X$ be an extremal blowup that is crepant with respect to 
$(X,c\MMM)$. 
Apply the construction of Theorem~\ref{thm:link} to $(\tilde{X},c \tilde{\MMM})$. 
By the above, $1/2\ge c=\alpha/\beta_6$, so 
\begin{equation}
\label{eq:beta-m}
\beta_6\ge 2\alpha.
\end{equation} 

The relation \eqref{eq:main} 
in our case has the form
\begin{equation}
\label{eq:formula}
6\hat q= 13 s_6+(13 \beta_6-6\alpha)e\ge13 s_6+ 20\alpha e.
\end{equation}
Consider the possibilities according to the center of the blowup $f(E)$. 

\begin{lemma}
\label{case:nG}
$f(E)$ is a non-Gorenstein point of $X$.
\end{lemma}
\begin{proof}
Assume that $f(E)$ is either a curve or a Gorenstein point.
Then $\alpha$ and $\beta_6$ are integers. It follows from \eqref{eq:formula} that $\hat q>3$.
Hence the contraction $\bar f$ is birational. Then $s_6>0$ because $\dim \MMM_6>0$ and so $\hat q\ge 6$.
Since $\hat q\le 19$, we have $e\alpha\le 5$.
On the other hand, $\hat q+ \alpha e\equiv 0\mod 13$.
The only possibility is $\hat q+ \alpha e=13$ and $s_6+\beta_6 e=6$. 
Since $\beta_6\ge 2\alpha$, we have $\alpha e\le 2$.
We obtain 
$\hat q=11$, $\alpha e= 2$. In this case $\Cl(\hat{X})$ is torsion free. This contradicts 
Corollary~\ref{cor:torsion}.
\end{proof}

\begin{scorollary}
\label{cor:discr}
If $f(E)$ is a point of index $r\neq 3$, then $\alpha=1/r$.
If $f(E)$ is a point of index $3$, then $\alpha=1/3$ or $2/3$.
\end{scorollary}
\begin{proof}
Recall that $\B(X)=(2,3,3,5,7)$.
Hence a point on $X$ of index $r\neq 3$ must be a cyclic
quotient singularity. Then $\alpha=1/r$ by \cite{Kawamata:Div-contr}.
If $f(E)$ is a point of index $3$, then by
\cite[Theorems~1.2-1.3]{Kawakita:hi} for $\alpha$ there are two possibilities:
$\alpha=1/3$ and $2/3$. 
\end{proof}

\begin{scorollary}
\label{cor:g}
If $\bar f$ is birational, then $\g(\hat X)\ge 4$. 
\end{scorollary}
\begin{proof}
Follows from Lemma~\ref{lemma:g} and Corollary~\ref{cor:discr}.
\end{proof}

\begin{lemma}
\label{case:P3}
$f(E)$ is not a point of index $3$.
\end{lemma}
\begin{proof}
Assume that $f(E)$ is a point of index $3$.
By Corollary \ref{cor:discr}
$\alpha=1/3$ or $2/3$. 
Since $\MMM$ is Cartier near $f(E)$, the number $\beta_6$ is an integer.
By \eqref{eq:beta-m}\ $\beta_6\ge 2\alpha$.
Then it is easy to see from \eqref{eq:formula} that 
$\hat q\ge 4$. Hence, $\bar f$ is birational 
and so $s_6>0$ because $\dim \MMM>0$.
One can check that there are only the following solutions:
$(\alpha, \hat q,e)=(2/3, 8,1)$, $(1/3, 4,1)$, and $(1/3, 8,2)$.
If $\hat q=8$, then $e\le 2$ and we get a contradiction by Corollary~\ref{cor:torsion} and Lemma~\ref{lemma:tor}\ref{lemma:tor-a}.
If $\hat q=4$, then $e=1$, $|\Clt(\hat{X})|\ge 3$ by Lemma~\ref{lemma:torsion}, and again this contradicts Lemma~\ref{lemma:tor}
because $\g(\hat X)\ge 4$ by Corollary~\ref{cor:g}.
\end{proof}

\begin{lemma}
\label{case:P7}
$f(E)\neq \rP_7$.
\end{lemma}
\begin{proof}
Assume that $f(E)=\rP_7$.
Then $\alpha=1/7$ by Corollary \ref{cor:discr} and
near $\rP_7$ we have $A_X\sim K_X$, $\MMM\sim -K_X$. 
Hence $\beta_6=1/7+m_6$, where $m_6\ge 1$ by \eqref{eq:beta-m}.
Then \eqref{eq:formula} can be rewritten as follows
\begin{equation*} 
6\hat{q}=e+13(s_6+m_6e).
\end{equation*}
The equation has the following solutions: $(\hat q, e)=(9,2)$ and $(11,1)$.
This contradicts Corollary~\ref{cor:torsion} and Lemma~\ref{lemma:tor}\ref{lemma:tor-a}.
\end{proof}

\begin{lemma}
\label{case:P2}
If $f(E)=\rP_2$, then \xref{prop:|6A|}\xref{prop:|6A|:P(1,2,3,5)} holds.
Moreover, $\ct(X, |6A_X|)=1/2$.
\end{lemma}
\begin{proof}
Assume that $f(E)=\rP_2$.
In this case $\alpha=1/2$ by Corollary \ref{cor:discr} and $\MMM$ is Cartier at $\rP_2$. Hence, $\beta_6$ is a positive integer.
Then \eqref{eq:formula} can be rewritten as follows
\begin{equation*}
3(2\hat q+e)= 13 (s_6+\beta_6 e).
\end{equation*}
If $\hat q\le 4$, then $e\le 2$ and $2\hat q+e\equiv 0\mod 13$,
a contradiction.
Therefore, $\hat q> 4$ and $\bar f$ is birational. Then $s_6>0$.
Assume that $e=1$. Then using $2\hat q+e\equiv 0\mod 13$ again we obtain $\hat q=6$ or $19$.
By Corollary~\ref{cor:torsion} \ $\Clt(X)\neq 0$.
This contradicts Lemma~\ref{lemma:tor}\ref{lemma:tor-a}.
Thus $e>1$ and $\hat q\ge 6$. Note that $e\le (6\hat q-13)/10$, so $e\le 10$ and $e<6\hat q/10$. 
Then using $2\hat q+e\equiv 0\mod 13$ again we obtain the following possibilities:
\[
(\hat q, e, s_6, \beta_6)=(11, 4, 2, 1)\quad \text{or}\quad (17, 5, 4, 1).
\]
If $\hat q=17$, then $\dim |4A_{\hat{X}}|\ge \dim \MMM=1$ because $s_6=4$.
On the other hand, $\hat X\simeq \PP(2,3,5,7)$ by Theorem~\ref{thm:indices}, so $\dim |4A_{\hat{X}}|=0$, a contradiction.
Therefore, $\hat q=11$ and $e=4$. Since $s_6=2$, 
$\dim |2A_{\hat{X}}|\ge \dim \MMM=1$. Then
$\hat{X}\simeq \PP(1,2,3,5)$ by \cite[Proposition~3.6 and Theorem~1.4]{P:2010:QFano}.
We obtain the case \ref{prop:|6A|}\ref{prop:|6A|:P(1,2,3,5)}.
Note that $\beta_6=1$, so $\ct(X, |6A_X|)=\alpha/\beta_6=1/2$ in this case.
\end{proof}

\begin{lemma}
\label{case:P5}
If $f(E)=\rP_5$, then \xref{prop:|6A|}\xref{prop:|6A|:P(1,1,2,3)} holds.
Moreover, $\ct(X, |6A_X|)=1/2$.
\end{lemma}
\begin{proof}
Assume that $f(E)=\rP_5$.
Then $\alpha=1/5$ by Corollary \ref{cor:discr} and
near $\rP_5$ we have $A_X\sim \MMM \sim -2K_X$, $4A_X\sim -3K_X$. 
Hence $\beta_4=3/5+m_5$.
For $k=4$ the relation \eqref{eq:main} has the form
\begin{equation*} 
4\hat{q}=13 s_4+(13 \beta_4-4 \alpha) e=7e+13(s_4+m_4e).
\end{equation*}
This gives the following possibilities:
\[
(\hat q, e) = (5,1), \quad (7,4),\quad  \text{and}\quad (17,6).
\]
If $(\hat q, e) = (5,1)$, then $|\Clt(\hat{X})|\ge 3$ by Corollary~\ref{cor:torsion}.
This contradicts Lemma~\ref{lemma:tor} because $\g(\hat{X})\ge 4$ by Corollary~\ref{cor:g}.

Assume that $(\hat q, e) =(17,6)$. By Theorem~\ref{thm:19-17} $\hat{X}\simeq \PP(2,3,5,7)$.
From \eqref{eq:main} for $k=3,4,7$ we obtain $s_3=3$, $s_4=2$, $s_7=5$.
The relations \eqref{eq:b-gamma-delta-1} have the form
\begin{equation*}
\begin{array}{lll}
6b&=&17\delta-13,
\\[1pt]
6\gamma_3&=&3\delta-3,
\\[1pt]
6\gamma_4&=&2\delta-4,
\\[1pt]
6\gamma_7&=&5\delta-7.
\end{array}
\end{equation*}

Since $\gamma_4\ge 0$,  $\delta\ge 2$. 
Hence, $b=(17\delta-13)/6>3$. 
Since $\hat{X}=\PP(1,1,2,3)$ has only cyclic quotient singularities, 
the point $\bar f(\bar{F})$ is smooth
by \cite{Kawamata:Div-contr}.
In particular, $b$ is an integer. 
But then $\delta\ge 5$.

Since $s_4=2$ and $\gamma_4>0$, $\bar f(\bar{F})\in \hat{M}_4\sim 2A_{\hat{X}}$.
Since $s_3=3$ and $\gamma_3>0$, $\bar f(\bar{F})\in \hat{M}_3\sim 3A_{\hat{X}}$.
Similarly,  $s_7=5$ and $\gamma_7>0$, hence $\bar f(\bar{F})\in \Bs \hat{\MMM}_7=|5A_{\hat{X}}|$.
Note that $\hat{M}_4,\, \hat{M}_3,\, \hat{M}_7$ are distinct prime divisors.
Hence in the homogeneous coordinates in $\hat{X}=\PP(2,3,5,7)$ we may assume that 
$\hat{M}_4=\{y_2=0\}$, $\hat{M}_3=\{y_3=0\}$, $\hat{M}_7=\{y_5=0\}$,
where $\deg(y_i)=i$.Thus 
the intersection $\hat{M}_4\cap \hat{M}_3\cap \hat{M}_7$ 
is the point $Q_7$ of index $7$ on $\hat{X}=\PP(2,3,5,7)$.
Hence, $\bar f(\bar{F})=Q_7$. 
On the other hand, by the above $\bar f(\bar{F})$ must be a smooth point of $\PP(2,3,5,7)$, a contradiction.

Therefore, $(\hat q, e) = (7,4)$. Then the relation \eqref{eq:main} for $k=7$ give us $s_7=1$, i.e. 
$\MMM_7\qq A_{\hat{X}}$. Since $\dim \MMM_7>0$, by Theorem~\ref{thm:19-17}\ $\hat{X}\simeq \PP(1,1,2,3)$.
We obtain the case \ref{prop:|6A|}\ref{prop:|6A|:P(1,1,2,3)}.
Applying \eqref{eq:main} with $k=6$ we obtain $s_6=2$ and $\beta_6=2/5$.
Therefore, $\ct(X, |6A_X|)=\alpha/\beta_6=1/2$.
\end{proof}
Thus we considered all the possibilities for the center of the blowup $f$.
This concludes the proof of Proposition~\ref{prop:|6A|}.
\end{proof}

\begin{proof}[Proof of Theorem~\xref{thm:|6A|}]
Let $f: \tilde{X}\to X$ be the Kawamata blowup of $\rP_5$.
Then $\alpha=1/5$.
By Proposition~\ref{prop:|6A|} $\ct(X,|6A_X|)=1/2$.
Therefore, $(X, \frac12 \MMM)$ is canonical.
As in the beginning of the proof of Proposition~\ref{prop:|6A|} we see that
$\beta_6\equiv 2 \alpha\mod \ZZ$. On the other hand, $1/2=\ct(X,|6A_X|)\le  \alpha/\beta_6$.
Hence, $\beta_6=2/5$ and  $\ct(X,|6A_X|)=  \alpha/\beta_6$. This means that
the extraction $f$ is $K_X+ \frac12 \MMM$-crepant. 
Then running the $K_{\tilde{X}}+ \frac12  \tilde{\MMM}$-MMP on $\tilde{X}$ 
we obtain the link \eqref{diagram-main}. According to 
Lemma~\ref{case:P5} \  $\hat{X}=\PP(1,1,2,3)$.
\end{proof}

\subsection{}
Now we consider the link of Theorem~\ref{thm:|6A|} in details.
For a divisor $\hat D$ on $\hat{X}$ or for a rational function $\varphi \in \Bbbk(\hat{X})^*$,
let $\upnu_{\bar{F}}(\hat D)$ (resp., $\upnu_{\bar{F}}(\varphi)$) denote the multiplicity of 
$\bar{f}^*(\hat D)$ (resp., $\varphi$) along $\bar{F}$.
Thus $\upnu_{\bar{F}}(\bar{M}_k)=\gamma_k$ and $\upnu_{\bar{F}}(\bar{E})=\delta$.

\begin{proposition}
We have $\bar{F}=\bar{M}_4$,
\begin{eqnarray*}
\label{eq:comp:mult1}
&
\hat{E}\sim 4A_{\hat{X}},\quad 
\hat{M}_3\sim A_{\hat{X}},\quad 
\hat{M}_7\sim A_{\hat{X}},\quad 
\hat{M}_6\sim 2A_{\hat{X}},\quad 
\hat{M}_5\sim 3A_{\hat{X}},
\\
\label{eq:comp:mult2}
&\upnu_{\bar{F}}(\hat{M}_3)=1, \quad
\upnu_{\bar{F}}(\hat{M}_7)=0, \quad
\upnu_{\bar{F}}(\hat{M}_6)=2, \quad
\upnu_{\bar{F}}(\hat{M}_5)=4.
\end{eqnarray*}
Moreover, $\bar f(\bar{F})$ is a smooth point of $\hat{X}$.
\end{proposition}
\begin{proof}
By our computations in \ref{case:P5}\ $e=4$. 
From \eqref{eq:main} for $k=3,5,6,7$ we obtain $s_3=s_7=1$, $s_6=2$, $s_5=3$.
Further, the relations \eqref{eq:b-gamma-delta-1} have the form
\begin{equation*}
\begin{array}{lll}
4b&=&7\delta-13,
\\[1pt]
4\gamma_3&=&\delta-3,
\\[1pt]
4\gamma_5&=&3\delta-5,
\\[1pt]
4\gamma_6&=&2\delta-6,
\\[1pt]
4\gamma_7&=&\delta-7.
\end{array}
\end{equation*}
Since $\gamma_7\ge 0$, we have $\delta\ge 7$. 
Hence, $b=(7\delta-13)/4\ge 9$. 
Since $\hat{X}=\PP(1,1,2,3)$ has only cyclic quotient singularities, 
the point $\bar f(\bar{F})$ is smooth
by \cite{Kawamata:Div-contr}. 
We claim that $\delta= 7$. 
Indeed, assume to the contrary that $\delta> 7$. Then $\gamma_7>0$, i.e. the point $\bar f(\bar{F})$ is contained in 
the base locus of the pencil $\hat{\MMM}_7=|A_{\hat{X}}|$. In other words, $\bar f(\bar{F})$ 
lies on the line $y_1=y_1'=0$. 
Since $s_6=2$ and $\gamma_6>0$, $\bar f(\bar{F})\subset \Bs \hat{\MMM}_6\subset |2A_{\hat{X}}|$.
Thus there exists an irreducible member $\hat{M}_6\in |2A_{\hat{X}}|$ passing through $\bar f( F)$.
This implies that $\bar f(\bar{F})=\{y_1=y_1'=y_2=0\}$, i.e. $\bar f(\bar{F})$ is the point of index $3$ 
on $\hat{X}=\PP(1,1,2,3)$, a contradiction. 
Thus $\delta= 7$
$\gamma_3=1$, $\gamma_5=4$, $\gamma_6=2$, $\gamma_7=0$.
\end{proof}

\begin{scorollary}
We can take coordinates $y_1,y_1',y_2,y_3$ in $\hat{X}=\PP(1,1,2,3)$ so that 
\[
\hat{M}_3=\{y_1=0\},\qquad  \hat{M}_7=\{y_1'=0\}, \qquad \hat{M}_6=\{y_2=0\}, \qquad \hat{M}_5=\{y_3=0\}.
\]
\end{scorollary}
\begin{proof}
Let $\psi_k$ be the equation of $\hat M_k$, $k=3,5,6,7$.
Since $\hat{M}_3\sim \hat{M}_7\sim A_{\hat{X}}$, the equations $\psi_3$ and $\psi_7$ have the form
$\lambda y_1+\lambda' y_1'=0$, where $\lambda$, $\lambda'$ are constants.
Applying linear coordinate change, we can put $\psi_3$ and $\psi_7$ to the desired form.
Similarly,
$\psi_6=\lambda_0 y_2+\lambda_1y_1^2+\lambda_2y_1y_1'+\lambda_3 y_1'^2$.
Since $\hat{M}_6$ is irreducible, $\lambda_0\neq 0$. 
Then the coordinate change
$y_2 \longmapsto y_2-\lambda_1y_1^2-\lambda_2y_1y_1'-\lambda_3 y_1'^2$ put 
$\psi_6$ to the desired form. The same trick works for
$\hat{M}_5$ if its equation contains the term $y_3$. Assume to the contrary that 
\[
 \psi_5=  \lambda y_2y_1+\lambda' y_2y_1'+\sum_{k=0}^3 \lambda_{k} y_1^k\, y_1'^{3-k}. 
\]
Since $\upnu_{\bar{F}}(\psi_5)=4$ and $y_1'^3$ is the only term with $\upnu_{\bar{F}}(y_1'^3)=0$, we have $\lambda_{0}=0$ and
similarly, $\lambda_{1}=0$.
Since  $y_2y_1'$ and $y_1^2\, y_1'$ are the only terms with $\upnu_{\bar{F}}(y_2y_1')$,
$\upnu_{\bar{F}}(y_1^2\, y_1')=2$, we have $\upnu_{\bar{F}}(\lambda' y_2y_1'+\lambda_{2} y_1^2\, y_1')\ge 3$.
Hence, $\upnu:=\upnu_{\bar{F}}(\lambda' y_2+\lambda_{2} y_1^2)\ge 3$.
Let $\hat D$ be the divisor given by $\lambda' y_2+\lambda_{2} y_1^2=0$
and let $\bar D$ (resp. $\tilde D$) be its proper transform on $\bar X$ (resp. $\tilde X$).
Then $\hat D\sim 2\hat M_7$, $\bar D\sim 2\hat M_7-\upnu\bar F= 2\hat M_7-\upnu\bar M_4$, and 
$\tilde D\sim 2\tilde M_7-\upnu\tilde M_4$. But then
$f_*\tilde D$ is an effective divisor such that $f_*\tilde D\sim 2(7-2\upnu) A$, $\upnu\ge 3$.
This contradicts Corollary~\ref{cor:dim}.
\end{proof}

For $m=3,\dots,7$, let $\varphi_m$ be a general section in $H^0(X,mA_X)$.
Once we fix a representative of $A_X$, the space $H^0(X,mA_X)$ has the following
presentation
\[
H^0(X,mA_X) =\lbraces \varphi\in \Bbbk(X)^* \mid mA_X +\di(\varphi)\ge 0 \rbraces \cup \lbraces 0\rbraces.
\]
Thus $\varphi_m$ can be regarded as rational functions on $X$.

\begin{scorollary}
\label{cor:map}
The map $\Psi_1: X \dashrightarrow \hat{X}=\PP(1,1,2,3)$ is given by
\[
P \longmapsto \big(\varphi_3(P)\varphi_4(P),\, \varphi_7(P),\, \varphi_6(P)\varphi_4^2(P),\, \varphi_5(P)\varphi_4^4(P)\big).
\]
\end{scorollary}

\begin{scorollary}
\label{cor:bir}
The map $\Psi: X \dashrightarrow \PP(3,4,5,6,7)$ given by
\[
P \longmapsto \big(\varphi_3(P),\, \varphi_4(P),\, \varphi_5(P),\, \varphi_6(P),\,\varphi_7(P)\big) 
\]
is birational onto its image.
\end{scorollary}

\begin{proof}
Indeed, the birational map $\Psi_1$ passes through $\Psi$.
\end{proof}

\section{Proof of Theorem~\xref{thm:main}}
Consider the graded algebra
\[
\R=\R(X)= \bigoplus_{m\ge 0}  H^0 (X, mA_X ).
\]
Since $A_X$ is ample, $X=\Proj (\R)$.
Let $\R'\subset \R$ be the graded subalgebra generated by 
$\varphi_3,\dots,\varphi_7$ and let $Y:=\Proj(\R')$. Then $Y$ is naturally embedded
to $\PP(3,4,5,6,7)$. More precisely, $Y\subset \PP(3,4,5,6,7)$ can be regarded 
as the image of the map
\[
\begin{array}{rcl}
\Psi: X&\longrightarrow& \PP(3,4,5,6,7),
\\[0.4em]
\phantom{ \Psi:} P &\longmapsto& (\varphi_{3}(P),\varphi_{4}(P),\varphi_{5}(P),\varphi_{6}(P),\varphi_{7}(P)).
\end{array}
\]
Since the map $\Psi$ is birational onto its image, $\dim Y=3$, 
so $Y$ is a hypersurface in $\PP(3,4,5,6,7) $.
In other words, 
\[
\R'\simeq \Bbbk[x_3,x_4,x_5,x_6,x_7]/(\Upsilon),
\]
where $\Upsilon$ is a quasi-homogeneous polynomial.
It is sufficient to show that $\R=\R'$.
Let $d:=\deg (\Upsilon)$.

For a graded algebra $\mathrm{S}$ by $\mathrm{S}_m$ we denote the 
homogeneous component of degree $m$.

\begin{claim}
$\deg (\Upsilon)\ge 10$. 
\end{claim}

\begin{proof}
Clearly, $d\ge 6$.
If $d=6$, then $\Upsilon$ has the form
$\Upsilon=a_1x_6+a_2x_3^2$. Hence, 
the relation $\varphi_6=\lambda \varphi_3^2$
holds in $\R$ for some $\lambda\in \Bbbk$.
But this contradicts Corollary~\ref{cor:dim}.

The cases $d=7$ and $d=8$
are considered similarly.
Let $d=9$. Then $\Upsilon$ has the form
$\Upsilon=a_1x_4x_5+a_2x_3x_6+a_3x_3^3$.
By Corollary~\ref{cor:map} the image of $Y$ in $\hat{X}=\PP(1,1,2,3)$
is contained in the surface $a_1z_3+a_2z_1z_2+a_3z_1^3=0$. This contradicts the birationality of $\Psi$. 
\end{proof}

\begin{claim}
$\deg (\Upsilon)=12$. 
\end{claim}

\begin{proof}
One can see that $\R_{12}$ contains 6 elements 
\[
\varphi_6^2,\quad
\varphi_5\varphi_7,\quad
\varphi_4^3,\quad
\varphi_3\varphi_4\varphi_5,\quad
\varphi_3^2\varphi_6,\quad
\varphi_3^4.
\]
Since $\dim \R_{12}=5$ by \eqref{prop:QFano13b},
these elements are linear dependent. Hence there is a non-trivial relation
$\Upsilon_1(\varphi_3, \dots, \varphi_7)=0$
of degree $12$ 
between $\varphi_3, \dots, \varphi_7$.
Since $Y$ is irreducible, its equation $\Upsilon$ divides $\Upsilon_1$.
Thus $\Upsilon_1=\Upsilon\, \Upsilon_2$ for some quasi-homogeneous polynomial $\Upsilon_2$ of degree $\le 2$.
On the other hand, degrees of all variables are $\ge 3$. Hence, $\Upsilon_2$ is a constant
and $\deg(\Upsilon)=12$.
\end{proof}
\begin{claim}
$\R'=\R$.
\end{claim}

\begin{proof}
Since $\R'\subset \R$, it is sufficient to show that $\dim \R'_m=\dim \R_m$
for all $m$, or equivalently, $\mathrm{p}_X(t)=\mathrm{p}_{\R'}(t)$.
By \eqref{prop:QFano13b}
\[ 
\mathrm{p}_X(t)=\frac{1-t^{12}}{(1-t^3)(1-t^4)(1-t^5)(1-t^6)(1-t^7)}.
\]
On the other hand, $\R'$ is the quotient of the polynomial ring 
by an ideal generated by one element.
Then it is easy to see that the Hilbert series $\mathrm{p}_{\R'}(t)$ of the algebra $\R'$ has the same form.
\end{proof}
Thus $\R'=\R$ and so $X=\Proj \left(\Bbbk[x_3,\dots,x_7]/(\Upsilon)\right)$, where $\deg(\Upsilon)=12$.
It remains to show that the equation $\Upsilon$
can be written in the form \ref{thm:main} \ref{thm:maina}  or \ref{thm:mainb}.

\begin{claim}
$x_5x_7\in \Upsilon$, $x_6^2\in \Upsilon$ and $x_4^3\in \Upsilon$.
\end{claim}
\begin{proof}
Indeed if $x_5x_7\notin \Upsilon$ the point $(0,0,0,0,1)\in X\subset \PP(3,4,5,6,7)$ is not a cyclic quotient singularity. Other assertions
are proved similarly. 
\end{proof}

Thus after an obvious coordinate change we may assume that $\Upsilon$ has the form
\[
\Upsilon=x_5x_7+x_4^3+ x_6^2+ \lambda_1 x_3^2x_6+ \lambda_1 x_3^4.
\]
Completing the square we may reduce the equation $\Upsilon$ to one of the desired forms~\ref{thm:main} \ref{thm:maina}  or~\ref{thm:mainb}.

\newcommand{\etalchar}[1]{$^{#1}$}
\def\cprime{$'$}

\end{document}